\documentclass[11pt,a4paper,twoside]{article}
\usepackage[english]{babel}
\usepackage[utf8x]{inputenc}
\usepackage{amsfonts}
\usepackage{amsmath}
\usepackage{amsthm}
\usepackage{amssymb}
\usepackage{pgfplots}
\usepackage{latexsym}
\usepackage{footmisc}
\usepackage{hyperref}
\usepackage{graphics}
\usepackage{enumerate}
\usepackage{graphicx}
\usepackage{fancyhdr}
\usepackage{tikz}
\usetikzlibrary{snakes}
\usetikzlibrary{calc,positioning}
\usetikzlibrary{3d,calc}
\usepgflibrary{shapes.geometric} 
\usetikzlibrary{calc}
\usepackage{wrapfig}
\usepackage{blindtext}
\newcommand{\mc}{\mathcal}
\newcommand{\mb}{\mathbb}

\pgfplotsset{my style/.append style={axis x line=middle, axis y line= middle, xlabel={$x$}, ylabel={$y$}, axis equal }}

\theoremstyle{plain}
\newtheorem{theorem}{Theorem}[section]
\newtheorem{lemma}[theorem]{Lemma}
\newtheorem{proposition}[theorem]{Proposition}

\theoremstyle{definition}

\newtheorem{remark}[theorem]{Remark}

\newtheoremstyle{named}{}{}{\itshape}{}{\bfseries}{.}{.5em}{\thmnote{#3's }#1}
\theoremstyle{named}

\begin{document}

\title{{\bf \large{TOPOLOGICAL GENERICITY OF NOWHERE DIFFERENTIABLE FUNCTIONS IN THE DISC AND POLYDISC ALGEBRAS} \vspace{4mm}}}
\author{Alexandros Eskenazis \and  Konstantinos Makridis}
\date{}
\maketitle

\begin{abstract}
\noindent In this paper we examine functions in the disc algebra $\mc{A}(D)$ and the polydisc algebra $\mc{A}(D^I)$, where $I$ is a finite or countably infinite set. We prove that, generically, for every $f \in \mc{A}(D)$ the continuous periodic functions $u=Ref|_{\mb{T}}$ and $\tilde{u} = Imf|_{\mb{T}}$ are nowhere differentiable on the unit circle $\mb{T}$. Afterwards, we generalize this result by proving that, generically, for every $f \in \mc{A}(D^I)$, where $I$ is as above, the continuous periodic functions $u=Ref|_{\mb{T}^I}$ and $\tilde{u} = Imf|_{\mb{T}^I}$ have no directional derivatives at any point of $\mb{T}^I$ and every direction $v \in \mb{R}^I$ \mbox{with $\|v\|_{\infty}=1$.}
\end{abstract}

\section{Introduction}

Strange functions have attracted the interest of mathematicians since Weierstrass, who first gave an explicit example of a function $u_0 : \mb{R} \to \mb{R}$ which is continuous, periodic but not differentiable at any real number. Since then there have been proven many results indicating the existence of functions with strange behaviour, such as non differentiability and universality in various senses.

In fact, many times these strange properties come to be generic, even if no explicit example of such a function is known. This is usually proven with some arguments involving Baire's Category Theorem in a suitable complete metric space or Fr\'echet space. For the role of Baire's Category Theorem in Analysis we refer to \cite{grosse} and \cite{kahane}. An old example of such a technique is the classical result of Banach and Mazurkiewicz (see \cite{ban}) which states that, generically, every continuous function on a compact interval $J$ of $\mb{R}$ is nowhere differentiable.

In section 2 of this paper we will prove an analogue of this theorem for functions in the disc algebra $\mc{A}(D)$, i.e. continuous functions defined on the closed unit disc $\overline{D}$ which are holomorphic on $D$. For every function $f \in \mc{A}(D)$ we can naturally construct a $2\pi-$periodic continuous function $h: \mb{R} \to \mb{C}$ defined by $h(\theta) = f(e^{i\theta})$; often, by abuse of notation, we write $f(\theta)$ instead of $h(\theta)$. Taking into consideration the above result of Banach and Mazurkiewicz it is natural to ask whether there are functions $f \in \mc{A}(D)$ such that the corresponding function $h$ is nowhere differentiable. Of course, since $h$ takes complex values, this means that for every point $\theta \in \mb{R}$, either $u=Reh$ will not be differentiable on $\theta$ or $\tilde{u} = Imh$ will not be differentiable on $\theta$. In Theorem 2.1, we will prove the stronger fact that, generically, for every function $f \in \mc{A}(D)$ both functions $u=Ref|_{\mb{T}}$ and $\tilde{u} = Imf|_{\mb{T}}$ are nowhere differentiable.

In section 3 we extend the above result in the context of several complex variables. A function $f: \overline{D}^I \to \mb{C}$ belongs to the algebra $\mc{A}(D^I)$ if it is continuous on $\overline{D}^I$, endowed with the cartesian topology, and separately holomorphic in $D^I$.  Equivalently, it is well known that $f \in \mc{A}(D^I)$ if and only if, $f$ is a uniform limit of polynomials on $\overline{D}^I$, where every polynomial is meant to depend only on finitely many variables. We consider functions $f \in \mc{A}(D^I)$, where $I$ is a finite or countably infinite set, such that neither their real part $u=Ref|_{\mb{T}^I}$, nor their imaginary part $\tilde{u}=Imf|_{\mb{T}^I}$ have directional derivatives at any point of $\mb{T}^I$, where $\mb{T}$ is the unit circle, for any direction $v \in \mb{R}^I$ with $\|v\|_{\infty}=1$. We prove that, generically, every $f \in \mc{A}(D^I)$ has the above properties. Finally, we prove that a direct generalization of this result cannot hold if $I$ is uncountable.

We mention that the previous result, valid for every $v \in \mb{R}^I$ with $\|v\|_{\infty}=1$, in the case where $I$ is a finite set implies the same result for every direction $v \in \mb{R}^I \setminus \{0\}$. Thus, in this case our result is the strongest possible. In the case where $I$ is infinite countable we have not been able to prove such a strong result. The set of directions with respect to which there is no directional derivative contains $\ell^{\infty}(I) \setminus \{0\}$ which is dense in $\mb{R}^I$, with respect to the cartesian topology.

The above results establish the topological genericity of nowhere differentiable functions in the disc and polydisc algebras. It remains an open question to determine other forms of genericity for this class of functions. For example, the dense lineability of this class, i.e. the existensc of a dense linear subspace $V$ of $\mc{A}(D)$ (or in general $\mc{A}(D^I)$), every non-zero element of which is nowhere differentiable. Also, it would be interesting to examine the spaceability of this class of functions, i.e. the existence of a closed, infinite dimensional linear subspace $W$ of $\mc{A}(D)$ (or in general $\mc{A}(D^I)$), every non-zero element of which is nowhere differentiable. Finally, since $\mc{A}(D)$ is an algebra, it would be interesting to study the algebrability of this class, i.e. to examine whether the above questions hold in the case where linear subspaces were replaced by subalgebras. These questions will hopefully be examined in future papers. For related results we refer to \cite{aron}, \cite{bayart} and \cite{gonz}.

\vspace{16mm}

\section{Nowhere differentiable functions in the disc algebra}

We start with the well known example of Weierstrass: the function $u_0 : \mb{R} \to \mb{R}$ defined by

\begin{equation}
u_0(x) = \sum_{n=0}^{\infty} a^n \cos(b^nx),
\end{equation}
for $0<a<1$ and an odd integer $b$, is a $2\pi-$periodic continuous function that is not differentiable at any real number $x$ if the condition $ab>1 + \frac{3\pi}{2}$ holds. In fact, for the differentiability of $u_0$ something even stronger holds (see \cite{titch}, pp. 351-354): for every $x \in \mb{R}$
\begin{equation}
\left| \limsup_{t \to x^+} \frac{u_0(t)-u_0(x)}{t-x} \right| = + \infty.
\end{equation}
The continuity of $u_0$ follows easily from Weierstrass' M-Test, since
$$\sum_{n=0}^{\infty} |a^n \cos(b^nx)| \leq \sum_{n=0}^{\infty} a^n < + \infty,$$
because $0<a<1$. Since $u_0$ is $2\pi-$periodic, we can see it as a function $u_0 : \mb{T} \to \mb{R}$, where $\mb{T}$ is the unit circle, and thus, it is well known that (see \cite{ahl}, pp. 168-171) there is a continuous extension of $u_0$ on the closed unit disc $\overline{D}$ which in harmonic on $D$. Of course, from the form of $u_0$ we can easily see that
\begin{equation}
u_0(z) = Re \left( \sum_{n=0}^{\infty} a^n z^{b^n} \right), \ \ \ |z| \leq 1.
\end{equation}
Hence, using once again the M-Test, the harmonic conjugate $\tilde{u}_0$ of $u_0$ can be extended continuously on $\overline{D}$ and its restriction on $\mb{T}$ is given by the formula
\begin{equation}
\tilde{u}_0(x) = \sum_{n=0}^{\infty} a^n \sin(b^nx).
\end{equation}
Using these elementary results we will provide a simple proof of the following:

\begin{theorem}
There is a $2\pi-$periodic continuous function $u: \mb{R} \to \mb{R}$ such that:
\begin{enumerate} [(i)]
\item $u$ is nowhere differentiable on $\mb{R}$.
\item It's harmonic conjugate $\tilde{u}$ extends continuously on $\overline{D}$ and
\item The $2\pi-$periodic continuous functions $\tilde{u}|_{\mb{T}}$ is nowhere differentiable on $\mb{R}$.
\end{enumerate}
In fact, the class of functions $f \in \mc{A}(D)$, such that the real part of $f$ has the above properties is residual, i.e. it contains a $G_{\delta}$ dense subset of $\mc{A}(D)$.
\end{theorem}

Our proof makes use of the Weierstrass function defined above and its properties. A more technical proof without use of Weierstrass's function $u_0$ can be found in \cite{eske}. We will need some lemmas. First, for $n \in \mb{N}$, consider the sets
\begin{multline}
D_n = \Big\{ u \in \mc{C}_{\mb{R}}(\mb{T}) : \mbox{for every} \ \theta \in \mathbb{R} \ \mbox{there is a } y \in \left( \theta, \theta + \frac{1}{n} \right) \\ \mbox{such that } \ |u(y) - u(\theta)| > n | y -\theta| \Big\},
\end{multline}
where, as usual, we interpret functions defined on $\mb{T}$ as $2\pi-$periodic functions defined on $\mb{R}$. Afterwards, we will also need the sets
\begin{equation}
E_n = \{ f \in \mc{A}(D) : Ref|_{\mb{T}} \in D_n \}, \ \ \ n \in \mb{N}.
\end{equation}

\begin{lemma}
For every $n \in \mb{N}$, $E_n$ is an open set in $\mc{A}(D)$, endowed with the supremum norm.
\end{lemma}

\begin{proof}
Let $n \in \mathbb{N}$ be fixed. We will prove that $\mc{A}(D) \setminus E_n$ is closed in $\mc{A}(D)$. Let $\{f_m\}$ be a sequence in $\mc{A}(D) \setminus E_n$ and $f \in \mc{A}(D)$ such that $f_m \to f$ uniformly on $\overline{D}$. Since $f_m \notin E_n$, for each m, there is a $\theta_m \in \mb{R}$ such that 
\begin{equation}
\left| \frac{u_m(y) - u_m(\theta_m)}{y-\theta_m} \right| \leq n,
\end{equation}
for every $y \in \left( \theta_m, \theta_m + \frac{1}{n} \right)$, where $u_m=Ref_m|_{\mb{T}}$. Since each $u_m$ is $2\pi-$periodic, we can assume that $\theta_m \in [0,2\pi]$ for every $m$ and hence there is a subsequence $\{ \theta_{k_m} \}$ of $\{ \theta_m \}$ and a $\theta \in [0,2\pi]$ such that $\theta_{k_m} \to \theta$.

\smallskip

If $y \in \left( \theta, \theta + \frac{1}{n} \right)$, then for large enough $m$ it holds $y \in \left( \theta_{k_m}, \theta_{k_m} + \frac{1}{n} \right)$. Thus, applying (7) for these indices $\{k_m\}$ and then letting $m \to \infty$ we have that
\begin{equation}
\left| \frac{u(y) - u(\theta)}{y-\theta} \right| \leq n,
\end{equation}
since the convergence of $\{f_m\}$ to $f$ is uniform (again $u=Ref|_{\mb{T}}$). Hence (8) holds for every $y \in \left( \theta, \theta + \frac{1}{n} \right)$ and thus $f \notin E_n$. So, $\mc{A}(D) \setminus E_n$ is closed or equivalently $E_n$ is open.
\end{proof}

Therefore, the intersection $\bigcap_{n=1}^{\infty} E_n$ is a $G_{\delta}$ set in $\mc{A}(D)$. 

\begin{lemma}
The set
\begin{equation}
\mc{S} = \bigcap_{n=1}^{\infty} E_n
\end{equation}
is dense in $\mc{A}(D)$.
\end{lemma}

\begin{proof}
Let $u_0 : \mb{T} \to \mb{R}$ be the Weierstrass function mentioned above, $\tilde{u}_0$ its harmonic conjugate and $f_0 \in \mc{A}(D)$ the function $f_0 = u_0 + i \tilde{u}_0$ defined on $\overline{D}$. From (2), we can deduce that for every $\theta \in \mb{R}$ and $n \in \mb{N}$ we can find $y_n = y_n(\theta) \in \left( \theta, \theta + \frac{1}{n} \right)$ such that
\begin{equation}
\left| \frac{u_0(y_n)-u_0(\theta)}{y_n-\theta} \right| > n,
\end{equation}
i.e. $u_0 \in D_n$ and thus $f_0 \in E_n$ for every $n \in \mb{N}$. Hence, $\mc{S} \neq \emptyset$.

\smallskip

In order to prove that $\mc{S}$ is dense in $\mc{A}(D)$ consider a polynomial $p(z)$ restricted on $\overline{D}$ and we will prove that $f_0 + p$ is also in $\mc{S}$. Let $n \in \mb{N}$ be fixed. The function $u=Rep|_{\mb{T}}$ is a $C^{\infty}$ real valued $2\pi-$periodic function defined on $\mb{R}$ and thus the same holds for its derivative $u'$. In particular, $u'$ is bounded by a constant $M>0$. Also, for every $y, \theta \in \mb{R}$ with $y \neq \theta$ there is a $\xi$ between $y$ and $\theta$ such that
\begin{equation}
\frac{u(y)-u(\theta)}{y-\theta} = u'(\xi).
\end{equation}
Hence, for every $y, \theta \in \mb{R}$, $y \neq \theta$ we have
\vspace{1mm}
\begin{equation}
\left| \frac{u(y)-u(\theta)}{y-\theta} \right| \leq M.
\end{equation}
\vspace{1mm}
Consider now a large $N>0$ such that $N>M+n$ and some $\theta \in \mb{R}$. Since $u_0 \in D_N$ there is a $y \in \left( \theta,  \theta + \frac{1}{N} \right)$ such that
\vspace{1mm}
\begin{equation}
\left| \frac{u_0(y)-u_0(\theta)}{y-\theta} \right| > N.
\end{equation}
\vspace{1mm}
Of course, it is true that $\theta < y < \theta + \frac{1}{n}$ and hence we have
\vspace{0.1mm}
\begin{equation}
\begin{split}
\left| \frac{Re(f_0+p)(y) - Re(f_0+p)(\theta)}{y-\theta} \right| & = \left| \frac{u_0(y)-u_0(\theta)}{y-\theta} + \frac{u(y)-u(\theta)}{y-\theta} \right| \\[.3cm]
& \geq \left| \frac{u_0(y)-u_0(\theta)}{y-\theta} \right| - \left| \frac{u(y)-u(\theta)}{y-\theta} \right| \\[.2cm]
& > N - M > n.
\end{split}
\end{equation}
Hence we conclude that $f_0+p \in E_n$ for the arbitrary $n$ and thus $f_0+p \in \mc{S}$. It is well known that the set of polynomials is dense in $\mc{A}(D)$; it follows that the set of translations
\begin{equation}
\{ f_0 + p : \ p \ \mbox{polynomial} \} \subseteq \mc{S}
\end{equation}
is also dense. From this we derive the density of $\mc{S}$ in $\mc{A}(D)$.
\end{proof}

\begin{lemma}
The set
\begin{equation}
\mc{T} = \mc{S} \cap i \mc{S} = \bigcap_{n=1}^{\infty} \big( E_n \cap iE_n \big)
\end{equation}
is also $G_{\delta}$ and dense in $\mc{A}(D)$.
\end{lemma}

\begin{proof}
Since multiplication by $i$ is a homeomorphism and $\mc{S}$ is $G_{\delta}$ and dense in $\mc{A}(D)$, it follows that $i\mc{S}$ is also $G_{\delta}$ and dense in $\mc{A}(D)$. Hence, from Baire's Category Theorem their intersection $\mc{T}$ is also $G_{\delta}$ and dense in $\mc{A}(D)$.
\end{proof}

We can now proceed to the proof of our main result:

\medskip

\noindent {\it Proof of Theorem 2.1.} Let $E$ the class of functions $f \in \mc{A}(D)$ such that both the real part $u$ and the imaginary part $\tilde{u}$ of $f$ are nowhere differentiable. We will prove that
\begin{equation}
\mc{T} = \bigcap_{n=1}^{\infty} \big( E_n \cap iE_n \big) \subseteq E
\end{equation}
and from this our result will follow. Let $f \in \mc{T}$. Since $f \in \mc{S}$ we conclude that $Ref = u \in D_n$ for every $n \in \mb{N}$. Hence, for $\theta \in \mb{T}$ and for every $n$ we can find a $y_n \in \left( \theta, \theta + \frac{1}{n} \right)$ such that
\begin{equation}
\left| \frac{u(y_n)-u(\theta)}{y_n-\theta} \right| > n.
\end{equation}
So $y_n \to \theta$ and thus, from (18), $u$ is not differentiable at $\theta$. For the imaginary part $\tilde{u}$ of $f$, since $f \in i\mc{S}$ we have that $-i f \in \mc{S}$ and thus $Re(-if) = Imf \in D_n$ for every $n \in \mb{N}$. It follows that $\tilde{u}$ is also nowhere differentiable on $\mb{T}$. So, the inclusion (17) has been proven and Theorem 2.1 now follows from Lemma 2.4.
$\hfill\Box$

\begin{remark}
Lemma 2.4 could be ommited if we had used some more difficult results. Hardy has proven in \cite{hardy} that if $0<a<1$ and $b \in \mb{N}$ such that $ab \geq 1$, then both functions
\begin{equation}
u_1(x) = \sum_{n=0}^{\infty} a^n \cos(b^nx) \ \ \ \mbox{and} \ \ \ \tilde{u}_1(x) = \sum_{n=0}^{\infty} a^n \sin(b^nx)
\end{equation}
are continuous and also belong in all the sets $D_n$ defined above with a much more technical and difficult proof from the one of Weierstrass. Hence, the function $f_1 : \overline{D} \to \mb{C}$ defined by
\begin{equation}
f_1(z) = \sum_{n=1}^{\infty} a^n z^{b^n}, \ \ \ |z| \leq 1
\end{equation}
is an element of $\mc{A}(D)$ such that $Ref_1=u_1$ has the properties listed in Theorem 2.1, i.e. $f_1 \in E$. Thus, using essentially the same density argument of Lemma 2.3 we could prove Theorem 2.1. It is of course obvious that the proof which we presented above is much more elementary.
\end{remark}

We close this section with a result concerning the size of the class of $2\pi-$periodic real valued functions $u$ that satisfy the properties listed in Theorem 2.1: we will prove that, in contrast with the result of Theorem 2.1, this subset of the space of $2\pi-$periodic functions is actually topologically small:

\begin{proposition}
Let $L$ be the subspace of $\mc{C}_{\mb{R}}(\mb{T})$ consisting of all continuous $2\pi-$periodic functions $u: \mb{R} \to \mb{R}$ which satisfy the properties (i), (ii) and (iii) of Theorem 2.1. Then $L$ is a dense set of first category in $\mc{C}_{\mb{R}}(\mb{T})$.
\end{proposition}

\begin{proof}
For the density of $L$, consider a function $u \in \mc{C}_{\mb{R}}(\mb{T})$ and an $\varepsilon >0$. We also denote by $u$ its harmonic extension on $D$ which is continuous on $\overline{D}$; due to the uniform continuity of $u$ on the compact $\overline{D}$, there is an $r \in (0,1)$ such that
\begin{equation}
|u(z) - u(rz)| < \frac{\varepsilon}{2},
\end{equation}
for every $z \in \overline{D}$. If $\tilde{u}$ is a harmonic conjugate of $u$, then the function $\varphi : \overline{D} \to \mb{R}$ defined by $\varphi(z) = u(rz)$ has a harmonic conjugate defined by the formula
\begin{equation}
\tilde{\varphi}(z) = \tilde{u}(rz),
\end{equation}
which extends continuously on $\overline{D}$ and thus there is a function $f \in \mc{A}(D)$ such that $\varphi = Ref$. Using the density of $\mc{T}$ proven above, we conclude that there is a function $g \in \mc{T}$ such that $\|f-g\|_{\infty} < \varepsilon/2$. If $\psi = Reg$, it holds that $\psi \in L$ and
\begin{equation}
\|u - \psi\|_{\infty} \leq \|u-\varphi\|_{\infty} + \|\varphi - \psi\|_{\infty} < \frac{\varepsilon}{2} + \|f-g\|_{\infty} < \varepsilon.
\end{equation}
Hence, $L$ is indeed dense in $\mc{C}_{\mb{R}}(\mb{T})$.

\smallskip

To prove that $L$ is of first category it suffices to prove that $L^c \equiv \mc{C}_{\mb{R}}(\mb{T}) \setminus L$ contains a $G_{\delta}$ dense subset. Of course, for every function $u \in L$, the harmonic conjugate $\tilde{u}$ is bounded on $D$. Hence, for the set
\begin{equation}
\begin{split}
\mc{Z} & = \{ u \in \mc{C}_{\mb{R}}(\mb{T}) : \ \sup_{|z| < 1} | \tilde{u}(z) | = \infty \} \\
& = \bigcap_{n=1}^{\infty} \{ u \in \mc{C}_{\mb{R}}(\mb{T}) : \ \sup_{|z| < 1} | \tilde{u}(z) | > n \}
\end{split}
\end{equation}
we conclude that $\mc{Z} \subseteq L^c$. We will prove that $\mc{Z}$ is $G_{\delta}$ and dense in $\mc{C}_{\mb{R}}(\mb{T})$.

\smallskip

In order to prove that $\mc{Z}$ is $G_{\delta}$ we will prove that for every $n$, the set
$\{ u \in \mc{C}_{\mb{R}}(\mb{T}) : \ \sup_{|z| \leq 1} | \tilde{u}(z) | > n \}$ is open in $\mc{C}_{\mb{R}}(\mb{T})$. This follows easily from the fact that for every $z \in D$, the linear operator
\begin{equation}
\mc{C}_{\mb{R}}(\mb{T}) \ni u \ \mapsto \ \tilde{u}(z) \in \mb{R}
\end{equation}
is continuous. Indeed, it is well known (see \cite{ahl} pp. 169) that
\begin{equation}
| \tilde{u}(z) | \leq \frac{1+|z|}{1-|z|} \|u\|_{\infty},
\end{equation}
from which the result follows.

\smallskip

For the density of $\mc{Z}$, we will proceed in a similar way as in the proof of Theorem 2.1. Using essentially the same argument, we can see that if $\mc{Z} \neq \emptyset$ then $\mc{Z}$ is dense in $\mc{C}_{\mb{R}}(\mb{T})$. Indeed, if $u \in \mc{Z}$, then for every real trigonometric polynomial $p: \mb{T} \to \mb{R}$ it will also be true that
\begin{equation}
\sup_{|z| < 1} |\tilde{u}(z) + \tilde{p}(z)| = \infty
\end{equation}
because $\tilde{p}$ is bounded on $\overline{D}$. So the problem is to find an element $u \in \mc{Z}$. Consider the domain $\Omega$ in $\mb{C}$ bounded by the line segments $x=-1$, $x=1$ the $x-$axis and the graph of the function $\varphi(x) = \frac{1}{x^2}$, i.e.
\begin{equation}
\Omega = \left\{ (x,y) \in \mb{R}^2 : 0<|x|<1 \ \mbox{and} \ 0 < y < \frac{1}{x^2} \right\} \cup \{ (0,y) : y>0 \}.
\end{equation}
Of course $\Omega$ is simply connected and not equal to the whole complex plane and thus, from the Riemann Mapping Theorem, there exists a conformal mapping $f: D \to \Omega$. Let $u=Ref$ and $\tilde{u} = Imf$. We will prove that $u \in \mc{Z}$.

\smallskip

Consider the map $\psi : \overline{\Omega} \cup \{ \infty \} \to \mb{C}$ defined by $\psi(z) = \frac{1}{z+i}$, which is a conformal map from $\Omega$ to $\psi(\Omega)$. Then, the image of $\partial \Omega \cup \{ \infty \}$ under $\psi$ is a Jordan curve in $\mb{C}$. Thus, the function $\psi \circ f : D \to \psi ( \Omega )$ is also a coformal map from the unit disc onto a Jordan domain. Hence, from Caratheodory's Theorem (see \cite{koos} pp. 36-40), $\psi \circ f$ has a continuous extension from $\overline{D}$ to $\overline{\psi(\Omega)}$ which is a homeomorphism and maps $\mb{T}$ on $\partial \psi(\Omega) = \psi \left( \partial \Omega \cup \{ \infty \} \right)$. Thus, $f$ maps continuously the unit circle $\mb{T}$ except one point to $\partial \Omega$ and the limit of $u$ at this point is $0$.

\smallskip

Hence, $u$ can be continuously extended on $\overline{D}$ and thus $u|_{\mb{T}}$ will be continuous and take values inside the interval $[-1,1]$. In contrast, the harmonic conjugate $\tilde{u}$ of $u$, restricted on the open disc $D$, will take every value of the interval $(0,+\infty)$ and thus it is unbounded. In other words $u \in \mc{Z}$ and the proof is complete.
\end{proof}

\section{The case of several variables}

We will now extend the results of the preceding section in the context of finitely many as well as infinite countably many complex variables. Let $I$ be a finite or countably infinite set. We will consider functions defined on the closed polydisc
\begin{equation}
\overline{D}^I = \{ (z_i)_{i \in I} : \ |z_i| \leq 1, \ \mbox{for every } i \in I \}
\end{equation}
which belong in the algebra of this polydisc $\mc{A}(D^I)$, i.e. they are holomorphic in $D^I$ and continuous on $\overline{D}^I$, where $\overline{D}^I$ is endowed with the cartesian topology. The distinguished boundary of this polydisc is the set $\mb{T}^I$ and a function $F : \mb{T}^I \to \mb{C}$ can equivalently be considered as a function $F: \mb{R}^I \to \mb{C}$ which is $2\pi-$periodic in every principal direction. We are interested in functions $f \in \mc{A}(D^I)$ such that their restriction $f|_{\mb{T}^I}$ has no directional derivatives in any point of $\mb{T}^I$ for a large set of directions, the largest we can prove. The main theorem of the preceding section takes the following form:

\begin{theorem}
Let $I$ be a finite or countably infinite set. There is function $f \in \mc{A}(D^I)$ such that for every point $\theta \in \mb{T}^I$ and every direction $v \in \mb{R}^I$ with $\|v\|_{\infty}=1$ both the functions $u=Ref|_{\mb{T}^I}$ and $\tilde{u} = Imf|_{\mb{T}^I}$ are not differentiable at $\theta$ in the direction $v$. In fact, the class of such functions $f$ is residual in $\mc{A}(D^I)$.
\end{theorem}

\noindent Following the steps of Section 2, for $n \in \mb{N}$ we consider the sets
\begin{equation}
\begin{split}
 D_n  = & \Big\{ u \in \mc{C}_{\mb{R}}(\mb{T}^I)  : \  \mbox{for every} \  \theta \in \mathbb{R}^I  \ \mbox{and every direction } v \in \mb{R}^I \ \mbox{with } \\ & \|v\|_{\infty}=1 \ \mbox{there is a }  y \in \left( \theta - \frac{1}{n} \cdot v, \theta + \frac{1}{n} \cdot v \right)  \mbox{such that } \\ & |u(y) - u(\theta)| > n \| y -\theta \|_{\infty} \Big\},
\end{split}
\end{equation}
and
\begin{equation}
E_n = \{ f \in \mc{A}(D^I) : Ref|_{\mb{T}^I} \in D_n \}.
\end{equation}
Since in the definition of $D_n$ we are interested in directions $v \in \mb{R}^I$ with $\|v\|_{\infty} = 1$ one can easily see that
\begin{equation}
D_n = \bigcap_{k \in I} D_n^{(k)},
\end{equation}
where each $D_n^{(k)}$ is defined as follows:
\begin{equation}
\begin{split}
D_n^{(k)}  = &  \Big\{ u \in \mc{C}_{\mb{R}}(\mb{T}^I)  :  \mbox{for every} \ \theta \in \mathbb{R}^I  \ \mbox{and every direction } v=(v_i)_{i \in I} \in \mb{R}^I \ \\ & \mbox{with } \|v\|_{\infty}=1 \  \mbox{and } |v_k| \geq \frac{1}{2} \ \mbox{there is a } y \in \left( \theta - \frac{1}{n} \cdot v, \theta + \frac{1}{n} \cdot v \right) \\ & \mbox{such that } |u(y) - u(\theta)| > n \| y -\theta \|_{\infty} \Big\}.
\end{split}
\end{equation}
Afterwards, we define naturally
\begin{equation}
E_n^{(k)} = \{ f \in \mc{A}(D^I) : Ref \in D_n^{(k)} \}.
\end{equation}

Proceeding as in the proof of Theorem 2.1 our objective is to show that the set
\begin{equation}
\bigcap_{n=1}^{\infty} \big( E_n \cap iE_n\big) = \bigcap_{n=1}^{\infty}\bigcap_{k \in I} \big( E_n^{(k)} \cap i E_n^{(k)} \big)
\end{equation}
is $G_{\delta}$ dense and its elements satisfy the desired properties of Theorem 3.1. We first prove that:

\begin{lemma}
For every $n \in \mb{N}$ and $k \in I$, $E_n^{(k)}$ is an open set in $\mc{A}(D^I)$, endowed with the supremum norm.
\end{lemma}

\begin{proof}
As in section 2, in order to prove that $\mc{A}(D^I) \setminus E_n^{(k)}$ is closed we consider a sequence $\{f_m\}$ such that $f_m \notin E_n^{(k)}$ for every $m$ and a function $f \in \mc{A}(D^I)$ such that $f_m \to f$ uniformly on $\overline{D}^I$. Since $f_m \notin E_n^{(k)}$, for each $m$, there is a $\theta^{(m)} \in \mb{T}^I$ and a direction $v^{(m)} = \big( v^{(m)}_i \big)_{i \in I} \in \mb{R}^I$ with $\|v^{(m)}\|_{\infty} = 1$ and $|v^{(m)}_k | \geq \frac{1}{2}$ such that
\begin{equation}
|u_m(y) - u_m(\theta^{(m)})| \leq n \|y-\theta^{(m)}\|_{\infty},
\end{equation}
 for every $y \in \left( \theta^{(m)} - \frac{1}{n} \cdot v^{(m)}, \theta^{(m)} + \frac{1}{n} \cdot v^{(m)} \right)$, where $u_m = Ref_m|_{\mb{T}^I}$. But both $\mb{T}^I$ and $\{ v=(v_i)_{i \in I} \in \mb{R}^I : \|v\|_{\infty} = 1 \ \mbox{and} \ |v_k| \geq \frac{1}{2} \}$ (equipped with the product topologies) are metrizable (since $I$ is at most countable) and compact  from Tychonoff's Theorem. Hence there is a strictly increasing sequence of positive integers $\{k_n\}$, a point $\theta \in \mb{T}^I$ and a direction $v=(v_i)_{i \in I} \in \mb{R}^I$ with $\|v\|_{\infty}=1$ and $|v_k| \geq \frac{1}{2}$ such that $\theta^{(k_m)} \to \theta$ and $v^{(k_m)} \to v$.

\smallskip

If $y \in \left( \theta - \frac{1}{n} \cdot v, \theta + \frac{1}{n} \cdot v \right)$ there is some $-\frac{1}{n} < s < \frac{1}{n}$ such that
\begin{equation} 
y = \theta + s \cdot v = \lim_{m \to \infty} \big( \theta^{(k_m)} + s \cdot v^{(k_m)} \big).
\end{equation}
Thus, applying (36) for these indices $\{k_m\}$, $y^{(k_m)} = \theta^{(k_m)} + s v^{(k_m)}$ and letting $m \to \infty$ we conclude that
\begin{equation}
|u(y)-u(\theta)| \leq n \|y-\theta \|_{\infty},
\end{equation}
since the convergence of $\{f_m\}$ to $f$ is uniform (again $u=Ref|_{\mb{T}^I}$). Hence (38) holds for every $y \in \left( \theta - \frac{1}{n} \cdot v, \theta + \frac{1}{n} \cdot v \right)$ and thus $f \notin E_n^{(k)}$. So $\mc{A}(D^I) \setminus E_n^{(k)}$ is closed or equivalently $E_n^{(k)}$ is open.
\end{proof}

Hence, the intersection $\bigcap_{n,k} E_n^{(k)}$ is a $G_{\delta}$ set in $\mc{A}(D^I)$.

\begin{lemma}
For every $k \in I$, the set
\begin{equation}
\mc{S}^{(k)} = \bigcap_{n=1}^{\infty} E_n^{(k)}
\end{equation}
is dense in $\mc{A}(D^I)$.
\end{lemma}

\begin{proof}
Fix some $k \in I$. We consider the function $f_k : \overline{D}^I \to \mb{C}$ defined by
\begin{equation}
f_k \big( (z_i)_{i \in I} \big) = f_0(z_k),
\end{equation}
where $f_0 \in \mc{A}(D)$ is the function defined in the proof of Lemma 2.3. It is obvious that $f_k \in \mc{A}(D^I)$; we will prove that $f_k \in \mc{S}^{(k)}$. Indeed, for some $\theta = (\theta_i)_{i \in I} \in \mb{T}^I$ and $v = (v_i)_{i \in I}$ with $\|v\|_{\infty}=1$ and $|v_k| \geq \frac{1}{2}$ the quotient
\begin{equation}
\left| \frac{Ref_k(\theta + tv) - Ref_k(\theta)}{t} \right| = \left| \frac{u_0(\theta_k + tv_k) - u_0(\theta_k)}{t} \right|
\end{equation}
does not remain bounded as $t \to 0$ because $v_k \neq 0$ and $u_0$ has the properties mentioned in the previous section. Hence, $\mc{S}^{(k)} \neq \emptyset$.

\smallskip

Next, as in Lemma 2.3, we will prove that for every complex polynomial $p \big( (z_i)_{i \in I} \big)$ the function $f_k + p$ is also in $\mc{S}^{(k)}$. We remind the reader that $p$ is a polynomial in a set of variables $(z_i)_{i \in I}$ if it is a polynomial in the set $(z_j)_{j \in J}$ for some finite subset $J \subseteq I$. Let $\theta \in \mb{T}^I$ and $v = (v_i)_{i \in I} \in \mb{R}^I$ with $\|v\|_{\infty}=1$ and $|v_k| \geq \frac{1}{2}$. The function $h : [-1,1] \to \mb{R}$ defined by
\begin{equation}
h(t) = u(\theta+tv),
\end{equation}
where $u = Rep|_{\mb{T}^I}$, is of course of class $C^{\infty}$ and thus its derivative $h'$ is bounded on $[-1,1]$ by a constant $M>0$. Thus, for every $y \in \left[ \theta - v, \theta + v \right]$, $y \neq \theta$ there is a $\xi \in (-1,1)$ such that
\begin{equation}
\frac{|u(y)-u(\theta)|}{\|y-\theta\|_{\infty}} = |h'(\xi)| \leq M.
\end{equation}
Thus, proceeding now as in the proof of Lemma 2.3 we can easily prove that $f_k + p \in \mc{S}^{(k)}$. Then, using the fact that the space of polynomials is dense in $\mc{A}(D^I)$ (see \cite{mak} for example) we conclude that also $\mc{S}^{(k)}$ is dense $\mc{A}(D^I)$.
\end{proof}

Combining now these two lemmas, we conclude that each $\mc{S}^{(k)}$ is a $G_{\delta}$ dense subset of $\mc{A}(D^I)$ and thus, since $I$ is at most countable, from Baire's Category Theorem (in the complete metric space $\mc{A}(D^I)$) the same holds for their intersection
\begin{equation}
\mc{S} = \bigcap_{k \in I} \mc{S}^{(k)} = \bigcap_{n=1}^{\infty} \bigcap_{k \in I} E_n^{(k)}.
\end{equation}
From this, the analogue of Lemma 2.4 follows immediately: the set
\begin{equation}
\mc{T} = \mc{S} \cap i \mc{S} = \bigcap_{n=1}^{\infty} \bigcap_{k \in I} \big( E_n^{(k)} \cap iE_n^{(k)} \big)
\end{equation}
is also a $G_{\delta}$ dense subset of $\mc{A}(D^I)$. With this result at hand we proceed now to the proof of the main result of this section:

\medskip

\noindent {\it Proof of Theorem 3.1.} Let $E$ be the class of functions $f \in \mc{A}(D^I)$ such that both the real part $u$ and the imaginary part $\tilde{u}$ of $f$ are nowhere differentiable in the sense of Theorem 3.1. We will prove that $\mc{T} \subseteq E$ and from this our result will follow.

\smallskip

Let $f \in \mc{T}$, a point $\theta \in \mb{T}^I$ and a direction $v \in \mb{R}^I$ with $\|v\|_{\infty} =1$. Since $f \in \mc{S}$, for every $n$ there is a $y_n \in \left( \theta - \frac{1}{n} \cdot v, \theta + \frac{1}{n} \cdot v \right)$, $y_n \neq \theta$ such that
\begin{equation}
\frac{|u(y_n)-u(\theta)|}{\|y_n - \theta\|_{\infty}} > n,
\end{equation}
where $u = Ref|_{\mb{T}^I}$. Hence $y_n \to \theta$ and (46) yields that $u$ is not differentiable at $\theta$ in the direction $v$. Using the fact that $f \in i \mc{S}$ we have the same conclusion for the function $\tilde{u} = Imf|_{\mb{T}^I}$ and the proof is complete.
$\hfill\Box$

\begin{remark}
If $I$ is an uncountable set the result of Theorem 3.1 cannot hold: it is true that if $f: \overline{D}^I \to \mb{C}$ is a continuous function, then $f$ depends only on a countable number of coordinates (see \cite{mak} for example). Thus, if $z_{j_0}$ is a coordinate which does not belong in this countable set then the partial derivative of $f$ with respect to $z_{j_0}$ will be of course zero at every point of $\mb{T}^I$.

\smallskip

Nevertheless, one can prove that generically for every $f \in \mc{A}(D^I)$ the following holds:

\smallskip

\noindent {\it For every $n \in \mb{N}$, there is a countably infinite subset $J_n \subseteq I$ such that for every point $\theta \in \mb{T}^I$ and every direction $v \in \mb{R}^I$ with $\|v\|_{\infty}=1$ and $v|_{J_n} \not\equiv 0$ there are points $y_1, y_2 \in \left( \theta - \frac{1}{n} \cdot v, \ \theta + \frac{1}{n} \cdot v \right)$ with}
\begin{equation}
|u(y_1) - u(\theta)| > n \|y_1-\theta\|_{\infty} \ \ \mbox{and} \ \ |\tilde{u}(y_2) - \tilde{u}(\theta) | > n \|y_2-\theta\|_{\infty}.
\end{equation}
{\it where $u=Ref|_{\mb{T}^I}$ and $\tilde{u} = Imf|_{\mb{T}^I}$.}

\smallskip

One can also prove that, the set of functions such that the above sets $J_n$ coincide is dense in $\mc{A}(D^I)$. However, this set may not be $G_{\delta}$, and thus this would not be a generic property of $\mc{A}(D^I)$.
\end{remark}

\bigskip

\noindent {\bf Aknowledgements:} Both authors would like to express their gratidute to Professor Vassili Nestoridis for presenting the problem as well as for his help and guidance throughout the creation of this paper.

\bigskip

\noindent {\scshape Alexandros Eskenazis:} Department of Mathematics, University of Athens, Panepistimioupolis, 157 84, Athens, Greece.\\
E-mail: \href{mailto:alex\_eske@hotmail.com}{\url{alex\_eske@hotmail.com}}

\medskip

\noindent {\scshape Konstantinos Makridis:} Department of Mathematics, University of Athens, Panepistimioupolis, 157 84, Athens, Greece.\\
E-mail: \href{mailto:kmak167@gmail.com}{\url{kmak167@gmail.com}}
\end{document}